\documentclass[11pt,bezier]{article}
\usepackage{amsmath,amssymb,amsfonts,euscript,graphicx}
\usepackage{amscd}
\usepackage{graphicx}

\newtheorem{prethm}{{\bf Theorem}}

\newenvironment{thm}{\begin{prethm}{\hspace{-0.5
               em}{\bf.}}}{\end{prethm}}

\newtheorem{prepro}[prethm]{{\bf Theorem}}

\newtheorem{preprop}[prethm]{{\bf Proposition}}

\newenvironment{prop}{\begin{preprop}{\hspace{-0.5
               em}{\bf.}}}{\end{preprop}}

\newtheorem{precor}[prethm]{{\bf Corollary}}

\newenvironment{cor}{\begin{precor}{\hspace{-0.5
               em}{\bf.}}}{\end{precor}}

\newtheorem{predefn}[prethm]{{\bf Definition}}

\newtheorem{preconj}[prethm]{{\bf Conjecture}}

\newtheorem{preremark}[prethm]{{\bf Remark}}

\newenvironment{remark}{\begin{preremark}\rm{\hspace{-0.5
               em}{\bf.}}}{\end{preremark}}

\newtheorem{preexample}[prethm]{{\bf Example}}

\newtheorem{prelem}[prethm]{{\bf Lemma}}

\newenvironment{lem}{\begin{prelem}{\hspace{-0.5
               em}{\bf.}}}{\end{prelem}}

\newtheorem{prelam}{{\bf Lemma}}

\newtheorem{preproof}{{\bf Proof.}}

\newenvironment{proof}[1]{\begin{preproof}{\rm
               #1}\hfill{$\Box$}}{\end{preproof}}

\title{\bf \large  The Center and Radius of the Regular Graph of Ideals 
\thanks {2010{ \it Mathematics Subject Classification}: 05C20, 05C69, 13E05, 16P20.}}

\author{{\normalsize {\sc F. Shaveisi}
}\vspace{3mm}\\
{\footnotesize{\it Department of Mathematics, Faculty of Sciences, Razi University, Kermanshah, Iran}}\\
{\footnotesize{ $\mathsf{f.shaveisi@razi.ac.ir}$}}}
\date{}
\begin{document}
\maketitle
\vspace{-.4cm}
\begin{abstract}
{\small \noindent The regular graph of ideals of the commutative ring $R$, denoted by ${\Gamma_{reg}}(R)$, is a graph whose vertex
set is the set of all non-trivial ideals of $R$ and two distinct vertices $I$ and $J$ are adjacent if and only if either $I$ contains a $J$-regular element or $J$ contains an $I$-regular element. In this paper,  it is proved that the radius of $\Gamma_{reg}(R)$ equals $3$. The central vertices of $\Gamma_{reg}(R)$ are determined, too.}
\end{abstract}

\vspace{5mm} \noindent{\bf\large 1. Introduction} \vspace{5mm}\\
We begin with recalling some definitions and notations on graphs. Let $G$ be a simple graph. The {\textit{distance}} between
two vertices $x$ and $y$ of $G$ is denoted by $d(x,y)$. A graph is said to be {\textit{connected}} if there exists a path between any two distinct vertices.
The {\textit{diameter}} of a connected graph $G$,
denoted by $diam(G)$, is the maximum distance between any pair of vertices of $G$.
For any vertex $x$ of a connected graph $G$, the {\it eccentricity} of
$x$, denoted by $e(x)$, is the maximum of the distances from $x$ to the other vertices of $G$.
The set of vertices with minimal eccentricity is called the {\textit{center}} of the graph, and
this minimum eccentricity value is the {\it radius} of $G$. Let $\Gamma$ be a digraph. An arc from a vertex $x$ to another vertex $y$ of $\Gamma$ is denoted by
$x\longrightarrow y$. Also we distinguish the {\it{out-degree}} $d_\Gamma^+(v)$, the number of edges leaving the vertex $v$, and the {\it{in-degree}} $d_\Gamma^-(v)$, the number of edges entering the vertex $v$. For more details about the standard
terminology of graphs, see \cite{bondy}.

Unless otherwise stated, throughout this
paper, all rings are assumed to be commutative Artinian rings with identity.
We
denote by ${Max}(R)$ and ${Nil}(R)$, the set of all maximal ideals and the set of all nilpotent
elements of $R$, respectively. The
ring $R$ is said to be \textit{reduced} if ${Nil}(R)=(0)$.
Also, the set of all zero-divisors of an $R$-module $M$
is denoted by $Z(M)$.
An element $r$ in the ring $R$ is called
\textit{$M$-regular} if $r\notin Z(M)$.  For every ideal $I$ of $R$, the {\it annihilator} of $I$ is denoted by $Ann(I)$.

As we know, most properties of a ring are closely tied to the
behavior of its ideals, so it is useful to study graphs or
digraphs, associated to the ideals of a ring. To see an instance of these
graphs, the reader is referred to \cite{disc.math,rocky,khashayarmanesh,MBI,intersection,actamathhungar,ShaveisiNikandish}.  The
{\it regular digraph of ideals} of a ring $R$, denoted by
$\overrightarrow{\Gamma_{reg}}(R)$, is a digraph whose vertex set
is the set of all non-trivial ideals of $R$ and for every two
distinct vertices $I$ and $J$, there is an arc from $I$ to  $J$ if
and only if $I$ contains a $J$-regular element. The underlying graph of $\overrightarrow{\Gamma_{reg}}(R)$, denoted by $\Gamma_{reg}(R)$, is called the {\it regular
graph} of ideals of $R$. For more information about this graph, see  \cite{khashayarmanesh,actamathhungar}. The main aim of this paper is to prove
that $r(\Gamma_{reg}(R))=3$.
\vspace{5mm} \noindent{\bf\large 2. Preliminary Results and Notation}

In this section, the distance between any pair of vertices of $\Gamma_{reg}(R)$ is determined, when $R$ is an Artinian non-reduced
ring and $\Gamma_{reg}(R)$ is a connected graph.
\begin{remark}\label{transitivity}
Let $I$, $J$ and $K$ be three distinct vertices of $\overrightarrow{\Gamma_{reg}}(R)$ and $I\longrightarrow J\longrightarrow K$ be a directed path in $\overrightarrow{\Gamma_{reg}}(R)$. Then by using the definition,
one can show that there is an arc from $I$ to $K$ in $\overrightarrow{\Gamma_{reg}}(R)$.
\end{remark}

The following notations are used all over this paper.

\noindent{\bf{Notation.}}
Let $I$
 and $K$ be (not necessarily distinct) ideals of $R$. We denote by
 $\mathcal{C}^-(I,K)$, the set of all non-trivial ideals $J$ of $R$ such that $J$ contains an $I$-regular element and $J$ contains a $K$-regular element. Also, we denote by
 $\mathcal{C}^+(I,K)$, the set of all non-trivial ideals $J$ of $R$ such that $I$ contains a $J$-regular element and $K$ contains a $J$-regular element. For simplicity, we denote $\mathcal{C}^-(I,(0))$ and $\mathcal{C}^+(I,R)$ by $\mathcal{C}^-(I)$ and $\mathcal{C}^+(I)$, respectively.



\begin{remark}\label{productadjacency}
Let $R_1,\ldots,R_n$ be rings, $R\cong R_1\times \cdots\times R_n$ and $I\cong I_1\times \cdots\times I_n$ and $J\cong J_1\times \cdots\times J_n$ be two distinct vertices of $\overrightarrow{\Gamma_{reg}}(R)$. Then
there is an arc from $I$ to $J$ if and only if $I_i\in\mathcal{C}^-_R(J_i)\cup \{R_i\}$, for every $i$, if and only if $J_i\in\mathcal{C}^+_R(I_i)\cup \{(0)\}$, for every $i$.
\end{remark}

Assume that $R$ is an Artinian ring such that $\Gamma_{reg}(R)$ is a connected  graph. Then by \cite[Theorem 2.3]{actamathhungar}, $|{\rm Max}(R)|\geq 3$ and $R$ contains a field as its
direct summand. So, \cite[Theorem 8.7]{ati} implies that $R\cong F_1\times R_2\times R_3$, where $F_1$ is a
field, $R_2$ is an Artinian local ring and $R_3$ is an Artinian ring. Moreover, if $R$ is non-reduced,
then we also can suppose that $R_3$ is not a field. Thus, any vertex of $\Gamma_{reg}(R)$ belongs to one of the following subsets:
\begin{enumerate}
\item[] $\mathfrak{A}=\{\mathfrak{a}=F_1\times R_2\times (0), \mathfrak{b}=F_1\times (0)\times R_3, \mathfrak{d}=(0)\times R_2\times R_3,\mathfrak{u}=(0)\times R_2\times (0), \mathfrak{v}=(0)\times (0)\times R_3\}$;
\item[] $\mathfrak{C}=\{\mathfrak{c}=F_1\times I_2\times I_3|\ I_2\times I_3\neq R_2\times R_3\}\setminus\{\mathfrak{a}, \mathfrak{b}\}$;
\item[] $\mathfrak{W}=\{\mathfrak{w}=(0)\times K_2\times K_3| \ K_2\times K_3\neq (0)\}\setminus\{\mathfrak{d},\mathfrak{u}, \mathfrak{v}\}$,
\end{enumerate}
where $I_2$ and $K_2$ are ideals of $R_2$ and $I_3$ and $K_3$ are ideals of $R_3$. From now, we use the above notations, to determine the distance between any disjoint pair of vertices of $\Gamma_{reg}(R)$. Also, for every Artinian ring $R$, we denote by $n_F(R)$, the number of fields appeared in the decomposition of $R$ to a direct product of Artinian local rings.
\begin{prop}\label{distancesimple}
Let $R$ be an  Artinian non-reduced ring. If  $\Gamma_{reg}(R)$ is a connected graph, then the following statements hold:
\begin{enumerate}
\item[\rm(1)] $d(\mathfrak{a}, \mathfrak{b})=d(\mathfrak{u},\mathfrak{v})=d(\mathfrak{a},
\mathfrak{d})=d(\mathfrak{b},\mathfrak{d})=2$

\item[\rm(2)] $d(\mathfrak{a},\mathfrak{u})=d(\mathfrak{b},\mathfrak{v})=d(\mathfrak{d},\mathfrak{u})=
    d(\mathfrak{d},\mathfrak{v})=d(\mathfrak{d},\mathfrak{w})=1$, where $\mathfrak{w}\in\mathfrak{W}$.

\item[\rm(3)] $d(\mathfrak{a},\mathfrak{v})=d(\mathfrak{b},\mathfrak{u})=3$.

\item[\rm(4)] For every vertex $\mathfrak{c}=F_1\times I_2\times I_3\in\mathfrak{C}$, $d(\mathfrak{a},\mathfrak{c})\leq 2$. Moreover, the equality holds if and only if $I_2\neq R_2$ and $I_3\neq(0)$.

\item[\rm(5)] For every vertex $\mathfrak{c}=F_1\times I_2\times I_3\in\mathfrak{C}$, $d(\mathfrak{b},\mathfrak{c})\leq 2$. Moreover, the equality holds if and only if $I_2\neq (0)$ and $I_3\neq R_3$.

\item[\rm(6)] For every vertex $\mathfrak{w}=(0)\times K_2\times K_3\in\mathfrak{W}$, $d(\mathfrak{u},\mathfrak{w})\leq 2$. Moreover, the equality holds if and only if $K_2\neq R_2$ and $K_3\neq (0)$.

\item[\rm(7)] For every vertex $\mathfrak{w}=(0)\times K_2\times K_3\in\mathfrak{W}$, $d(\mathfrak{v},\mathfrak{w})\leq 2$. Moreover, the equality holds if and only if $K_2\neq (0)$ and $K_3\neq R_3$.
\end{enumerate}
\end{prop}
\begin{proof}
{The assertions (1), (2) and (3) are clear and follow from the definition.
%
%
%
%
Choose $\mathfrak{c}=F_1\times I_2\times I_3\in \mathfrak{C}$. Clearly, $\mathfrak{a}$ and $\mathfrak{c}$ are adjacent if and only if either $I_3=(0)$ or $I_2=R_2$. So, we can suppose that $I_3\neq (0)$ and $I_2\neq R_2$. Since both of vertices $\mathfrak{a}$ and $\mathfrak{c}$ are adjacent to $F_1\times (0)\times (0)$, we deduce that
    $$d(\mathfrak{a},\mathfrak{c})=
  \begin{cases}
   1;       & {\rm either} \ I_3=(0)\ {\rm or}\ I_2=R_2 \\
   2;       & I_3\neq (0) \ {\rm and}\ I_2\neq R_2.
  \end{cases}$$
  So, (4) follows. Also, the proofs of (5), (6) and (7) are similar to that of (4).
}
\end{proof}

\begin{prop}\label{distanceccww}
Let $R$ be an  Artinian non-reduced ring. If  $\Gamma_{reg}(R)$ is a connected graph, then the following statements hold:
\begin{enumerate}
\item[\rm(1)] For every two distinct vertices $\mathfrak{c}=F_1\times I_2\times I_3$ and $\mathfrak{c}'=F_1\times J_2\times J_3$ in $\mathfrak{C}$, we have $d(\mathfrak{c},\mathfrak{c}')\leq 2$. Moreover, $d(\mathfrak{c},\mathfrak{c}')=1$ if and only if $I_2=I_3=(0)$, $J_2=J_3=(0)$ or $I_2\times I_3$ and $J_2\times J_3$ are adjacent in $\Gamma_{reg}(R_2\times R_3)$.

\item[\rm(2)] For every two distinct vertices $\mathfrak{w}=(0)\times K_2\times K_3$ and $\mathfrak{w}'=(0)\times L_2\times L_3$ in $\mathfrak{W}$, we have $d(\mathfrak{w},\mathfrak{w}')\leq 2$. Moreover,  $d(\mathfrak{w},\mathfrak{w}')=1$ if and only if $K_2\times K_3=R_2\times R_3$, $L_2\times L_3=R_2\times R_3$ or $K_2\times K_3$ and $L_2\times L_3$ are adjacent in $\Gamma_{reg}(R_2\times R_3)$.
\end{enumerate}
\end{prop}
\begin{proof}
{This is a direct consequence of the definition.
}
  \end{proof}
  \begin{prop}\label{distancec}
Let $R$ be an  Artinian non-reduced ring. If $\Gamma_{reg}(R)$ is a connected graph, then for every vertex $\mathfrak{c}=F_1\times I_2\times I_3\in\mathfrak{C}$, we have:
\begin{enumerate}
\item[\rm(1)]  $d(\mathfrak{c},\mathfrak{u})=
  \begin{cases}
   1;       & I_2=R_2 \\
   2;       &I_2\neq R_2\ {\rm and}\ \mathcal{C}^-_{R_3}(I_3)\neq\varnothing. \\
   3;       & {\rm Otherwise}.
  \end{cases}$
 \item[\rm(2)] $d(\mathfrak{c},\mathfrak{v})=
  \begin{cases}
   1;       &I_3=R_3\\
   2;       &  I_3\neq R_3\ {\rm and\ either}\ I_2=(0)\ {\rm or}\ \mathcal{C}^+(I_3)\neq\varnothing\\
   3;       &  {\rm otherwise}.
  \end{cases}$
  \item[\rm (3)]$d(\mathfrak{c},\mathfrak{d})=
  \begin{cases}
   2;       &   I_2=R_2, I_3=R_3\ {\rm or}\ \mathcal{C}^+_{R_3}(I_3)\neq\varnothing\\
   3;       &  d(\mathfrak{c},\mathfrak{d})\neq 2\ {\rm and }\ (I_2=(0), I_3=(0)\ {\rm or}\ \mathcal{C}^-_{R_3}(I_3)\neq\varnothing)\\
   4;       & {\rm Otherwise}.
  \end{cases}$
  \item[\rm(4)]
  $d(\mathfrak{a},\mathfrak{w})=
  \begin{cases}
   1;       &K_3=(0)\\
   2;       &  K_3\neq(0)\ {\rm and\ either}\ K_2=R_2\ {\rm or}\ \mathcal{C}^-(K_3)\neq\varnothing\\
   3;       & {\rm Otherwise}.\\
  \end{cases}$

 \item[\rm (5)] $d(\mathfrak{b},\mathfrak{w})=
  \begin{cases}
   1;       & K_2=(0) \\
   2;       & K_2\neq (0)\ {\rm and}\  \mathcal{C}^+(K_3)\neq\varnothing\\
   3;       & K_2\neq (0)\ {\rm and}\ \mathcal{C}^+(K_3)=\varnothing.\\
  \end{cases}$
\end{enumerate}
\end{prop}
\begin{proof}
{We only prove (3). The other assertions are proved, similarly.
It is clear that $\mathfrak{c}$ and $\mathfrak{d}$ are not adjacent. On the other hand, since $\mathfrak{u}$ and $\mathfrak{d}$ are adjacent, (1) implies
that $2\leq d(\mathfrak{c},\mathfrak{d})\leq 4$. Now, we follow the  proof in the following three steps:

Step 1. $d(\mathfrak{c},\mathfrak{d})=2$ if and only if $I_2=R_2$, $I_3=R_3$ or $\mathcal{C}^+_{R_3}(I_3)\neq \varnothing$:\\
If $I_2=R_2$($I_3=R_3$), then both $\mathfrak{c}$ and $\mathfrak{d}$ are adjacent with $\mathfrak{u}$($\mathfrak{v}$). Assume that $J_3\in\mathcal{C}^+_{R_3}(I_3)$. Then $\mathfrak{c}$ and $\mathfrak{d}$ are adjacent to $(0)\times (0)\times J_3$. Therefore, in any case, we have $d(\mathfrak{c},\mathfrak{d})=2$. Conversely, let $d(\mathfrak{c},\mathfrak{d})=2$. Then by Remark \ref{transitivity}, there exists a vertex $J=J_1\times J_2\times J_3$ such that one of the paths $\mathfrak{c}\longleftarrow J\longrightarrow\mathfrak{d}$ and $\mathfrak{c}\longrightarrow J\longleftarrow\mathfrak{d}$ exists. By using Remark \ref{productadjacency}, one can deduce that the existence of the first path is impossible. So, we can suppose that only the second path exists.
By Remark \ref{productadjacency}, $J_1=(0)$. Hence either $J_2\neq(0)$ or $J_3\neq (0)$.   Therefore, Remark \ref{productadjacency} implies that $I_2=R_2$, $I_3=R_3$ or  $J_3\in \mathcal{C}^+(I_3)\neq \varnothing$.

From now, suppose that $I_2\neq R_2$, $I_3\neq R_3$ and $\mathcal{C}^+_{R_3}(I_3)=\varnothing$. This means that $d(\mathfrak{c},\mathfrak{d})\geq 3$.

Step 2. $d(\mathfrak{c},\mathfrak{d})=3$ if and only if either $I_2=(0)$ or $\mathcal{C}^-(I_3)\neq\varnothing$:\\
By Remark \ref{transitivity}, $d(\mathfrak{c},\mathfrak{d})=3$ if and only if there exist two vertices, say $J=J_1\times J_2\times J_3$ and $L=L_1\times L_2\times L_3$, such that one of the following paths exists:
\begin{eqnarray}\label{05}
\mathfrak{c}=F_1\times I_2\times I_3 \longleftarrow J_1\times J_2\times J_3 \longrightarrow L_1\times L_2\times L_3 \longleftarrow (0)\times R_2\times R_3=\mathfrak{d};
\end{eqnarray}
\begin{eqnarray}\label{06}
\mathfrak{c}=F_1\times I_2\times I_3 \longrightarrow J_1\times J_2\times J_3\longleftarrow L_1\times L_2\times L_3 \longrightarrow (0)\times R_2\times R_3=\mathfrak{d};
\end{eqnarray}
By Remark \ref{productadjacency}, Path (\ref{05}) exists if and only if $J_1=F_1$ if and only if either $J_2\neq R_2$ or $J_3\neq R_3$. On the other hand, from Remark \ref{productadjacency}, we deduce that $J_2\neq R_2$ if and only if $I_2=(0)$ and $J_3\neq R_3$ if and only if either $I_3=(0)$ or $\mathcal{C}^-(I_3)\neq\varnothing$. To complete the proof, it is enough to show that Path (\ref{06}) does not exist. Since $I_2\neq R_2$ and $\mathcal{C}^+(I_3)=\varnothing$, we deduce that $J_1=J_2=(0)$. Thus Remark \ref{productadjacency} implies that $J=F_1\times (0)\times (0)$ and $L=(0)\times R_2\times R_3$ and this contradicts the adjacency of $J$ and $L$. Hence Path (\ref{06}) does not exist and so, we are done.
}
\end{proof}
\begin{prop}\label{distancecandw}
Let $R$ be an  Artinian non-reduced ring. If  $\Gamma_{reg}(R)$ is a connected graph, then for every $\mathfrak{c}=F_1\times I_2\times I_3\in\mathfrak{C}$ and every $\mathfrak{w}=(0)\times K_2\times K_3\in \mathfrak{W}$, we have:
  $$d(\mathfrak{c},\mathfrak{w})=
  \begin{cases}
   1;       &   I_2 \times I_3\  {\rm contains\ a}\ K_2\times K_3{\rm-regular\ element}\\
   2;       &   I_2=K_2=(0), I_2=K_2=R_2\ {\rm or}\ \mathcal{C}^-(I_3,K_3)\cup\mathcal{C}^+(I_3,K_3)\neq\varnothing\\
   5;       &  I_2, K_2\ {\rm are\ nontrivial,\ } \mathcal{C}^+(I_3)\cup\mathcal{C}^-(I_3)\cup\mathcal{C}^+(K_3)\cup\mathcal{C}^-(K_3)=\varnothing\\
   3\ {\rm or}\ 4;       & {\rm Otherwise}.
  \end{cases}$$
\end{prop}
\begin{proof}
{It is clear that $\mathfrak{w}$ contains no $\mathfrak{c}$-regular element. Hence $d(\mathfrak{c},\mathfrak{w})=1$ if and only if $\mathfrak{c}$ contains a $\mathfrak{w}$-regular element, say $(x_1,x_2,x_3)$, if and only if $(x_2,x_3)\in I_2\times I_3$ is a $K_2\times K_3$-regular element. Assume that $I_2\times I_3$ contains no $K_2\times K_3$-regular element. We claim that $d(\mathfrak{c},\mathfrak{w})=2$ if and only if $I_2=K_2=(0)$, $I_2=K_2=R_2$ or $\mathcal{C}^-(I_3,K_3)\cup\mathcal{C}^+(I_3,K_3)\neq\varnothing$. By Remark \ref{transitivity}, $d(\mathfrak{c},\mathfrak{w})=2$ if and only if there exists a vertex, say $J=J_1\times J_2\times J_3$, such that one of the following paths exists:
\begin{eqnarray}\label{11}
\mathfrak{c}=F_1\times I_2\times I_3 \longleftarrow J_1\times J_2\times J_3 \longrightarrow  (0)\times K_2\times K_3=\mathfrak{w};
\end{eqnarray}
\begin{eqnarray}\label{12}
\mathfrak{c}=F_1\times I_2\times I_3 \longrightarrow J_1\times J_2\times J_3\longleftarrow (0)\times K_2\times K_3=\mathfrak{w};
\end{eqnarray}
By Remark \ref{productadjacency}, Path (\ref{11}) exists if and only if $J_1=F_1$ if and only if either $J_2\neq R_2$ or $J_3\neq R_3$. On the other hand, from Remark \ref{productadjacency}, we deduce that $J_2\neq R_2$ if and only if $I_2=K_2=(0)$ and $J_3\neq R_3$ if and only if $\mathcal{C}^-(I_3,K_3)\neq\varnothing$. Also,  Path (\ref{12})  exists  if and only if $J_1=(0)$ if and only if either $J_2\neq (0)$ or $J_3\neq (0)$. Moreover, Remark \ref{productadjacency} implies that $J_2\neq (0)$ if and only if $I_2=K_2=R_2$ and $J_3\neq (0)$ if and only if
$\mathcal{C}^+(I_3,K_3)\neq\varnothing$. So, the claim is proved.
Finally, assume that $d(\mathfrak{c},\mathfrak{w})\geq 3$, $I_2$ and $K_2$ are non-trivial ideals of $R_2$ and $\mathcal{C}^-(I_3)\cup\mathcal{C}^+(I_3)\cup\mathcal{C}^-(K_3)\cup\mathcal{C}^+(K_3)=\varnothing$. Then \cite[Theorem 2.1]{actamathhungar} and Remark \ref{transitivity} imply that $d^+_{\overrightarrow{\Gamma_{reg}}(R)}(\mathfrak{w})=d^-_{\overrightarrow{\Gamma_{reg}}(R)}(\mathfrak{c})=0$. We show that $d(\mathfrak{c},\mathfrak{w})=5$. Suppose to the contrary, $d(\mathfrak{c},\mathfrak{w})=3$ or $4$. Then by Remark \ref{transitivity}, there exist three non-trivial ideals of $R$, say $J$, $L$ and $P$ such that one of the following paths exists:
\begin{eqnarray}\label{13}
F_1\times I_2\times I_3\longrightarrow J\longleftarrow L \longrightarrow (0)\times K_2\times K_3;
\end{eqnarray}
\begin{eqnarray}\label{14}
F_1\times I_2\times I_3\longleftarrow J\longrightarrow  L \longleftarrow(0)\times K_2\times K_3;
\end{eqnarray}
\begin{eqnarray}\label{15}
F_1\times I_2\times I_3\longrightarrow J\longleftarrow L \longrightarrow  P \longleftarrow(0)\times K_2\times K_3;
\end{eqnarray}
\begin{eqnarray}\label{16}
F_1\times I_2\times I_3\longleftarrow J\longrightarrow  L \longleftarrow P\longrightarrow(0)\times K_2\times K_3;
\end{eqnarray}
Since $d^+_{\overrightarrow{\Gamma_{reg}}(R)}(\mathfrak{w})=d^-_{\overrightarrow{\Gamma_{reg}}(R)}(\mathfrak{c})=0$, Paths (\ref{14}), (\ref{15}) and (\ref{16}) don't exist. Thus we can assume that Path (\ref{13}) exists. Then Remark \ref{productadjacency} implies that
$J=F_1\times (0)\times (0)$ and  $L=(0)\times R_2\times R_3$ which contradicts the adjacency of $J$ and $L$. Therefore, $d(\mathfrak{c},\mathfrak{w})=5$ and the proof is complete.}
\end{proof}
\vspace{2mm} \noindent{\bf\large 3. Main Results}

In this section,
it is proved that the radius of $\Gamma_{reg}(R)$ equals $3$. The central vertices are characterized, too. First we need some lemmas.
\begin{lem}\label{k1C+-}
Let $R$ be an Artinian ring which is not field. If $n_F(R)\geq 1$, then for every ideal $I$ of $R$,  $\mathcal{C}^+(I)\cup\mathcal{C}^-(I)\neq\varnothing$.
\end{lem}
\begin{proof}
{Since $R$ is an Artinian  ring, \cite[Theorem 8.7]{ati} implies that $R\cong F_1\times R_2$, where $F_1$ is a field and $R_2$ is an Artinian ring. For every ideal $I=I_1\times I_2$ of $R$, either $I_1=(0)$ or $I_1=F_1$. If $I_1=(0)$, then $(0)\times R_2\in \mathcal{C}^-(I)$. Also, if $I_1=F_1$, then $F_1\times (0)\in \mathcal{C}^+(I)$. Thus in any case, $\mathcal{C}^+(I)\cup\mathcal{C}^-(I)\neq\varnothing$.
}
\end{proof}

{\rm Let $R=R_1\times R_2\times\cdots\times R_n$ be an Artinian ring, where every $R_i$ is an Artinian local ring. For every ideal $I=I_1\times I_2\times\cdots\times I_n$ of $R$, setting
$$I_i^c=
  \begin{cases}
   R_i;       &  I_i=(0)\\
   (0);       & I_i=R_i\\
   I_i;       & I_i\ {\rm is\ a\ non-trivial \ ideal \ of}\ R_i,
     \end{cases}$$
we define the {\it complement} of $I$ to be $I^c=I_1^c\times I_2^c\times\cdots\times I_n^c$. Also, for every subset $X$ of ideals
of $R$, by $X^c$, we mean the set $\{I^c|\ I\in X\}$. }

\begin{lem}\label{eccentricity3}
Let  $R$ be an Artinian ring such that $\Gamma_{reg}(R)$ is a connected graph. Then for every vertex $I$ of $\Gamma_{reg}(R)$, $e(I)\geq 3$.
\end{lem}
\begin{proof}
{Let $R$ be an Artinian ring. Since $\Gamma_{reg}(R)$ is connected, \cite[Theorem 3.2]{actamathhungar} and \cite[Theorem 8.7]{ati} imply that $R\cong F_1\times R_2\times \cdots\times R_n$, where  $n\geq 3$, $F_1$ is a field and every $R_i$, $2\leq i\leq n$, is an Artinian local ring. So, $I=I_1\times I_2\times\cdots\times I_n$, where every $I_i$ is an ideal of $R_i$.
We show that  $d(I,I^c)\geq 3$. Suppose to the contrary, $d(I,I^c)\leq 2$. It is clear that $I$ and $I^c$ are not adjacent. Thus by Remark \ref{transitivity}, there exists a vertex, say $J=J_1\times J_2\times\cdots\times J_n$ such that one of the following paths exists:
\begin{eqnarray*}
I=I_1\times I_2\times\cdots\times I_n\longleftarrow J_1\times J_2\times\cdots \times J_n\longrightarrow I_1^c\times I_2^c\times\cdots\times I_n^c=I^c
\end{eqnarray*}
\begin{eqnarray*}
I=I_1\times I_2\times\cdots\times I_n\longrightarrow J_1\times J_2\times\cdots \times J_n\longleftarrow I_1^c\times I_2^c\times\cdots\times I_n^c=I^c
\end{eqnarray*}
If the first path exists, then Remark \ref{productadjacency} implies that for every $i$, $J_i$ contains an $I_i$-regular element and $J_i$ contains an $I_i^c$-regular element. Thus $J_i=R_i$, for every $i$, and hence $J=R$, a contradiction. Similarly, it is seen that the existence of the second path leads to a contradiction.
}
\end{proof}

\begin{thm}\label{abduv-eccentricity}
Let  $R$ be an Artinian non-reduced ring such that $\Gamma_{reg}(R)$ is a connected graph.
Then the following statements hold:\\
$\rm(i)$ $e(\mathfrak{a})=e(\mathfrak{b})=e(\mathfrak{u})=e(\mathfrak{v})=3$.\\
$\rm(ii)$ $e(\mathfrak{d})=3$ if and only if $n_F(R)\geq 2$.
\end{thm}
\begin{proof}
{(i) This follows from Propositions \ref{distancesimple}, \ref{distancec} and Lemma \ref{eccentricity3}.\\
(ii) By Proposition \ref{distancesimple} and Lemma \ref{eccentricity3}, it is enough to check $d(\mathfrak{d},\mathfrak{c})$, where $\mathfrak{c}=F_1\times I_2\times I_3\in \mathfrak{C}$. First suppose that $n_F(R)\geq 2$. If $R_3$ is not field, then Lemma \ref{k1C+-} yields that $\mathcal{C}^+(I_3)\cup \mathcal{C}^-(I_3)\neq \varnothing$, for every ideal $I_3$ of $R_3$. Thus by Proposition \ref{distancec} (3), $d(\mathfrak{d},\mathfrak{c})\leq 3$.  Also, if $R_3$ is a field, then $I_3$ is a trivial ideal of $R_3$ and so again by Proposition \ref{distancec} (3), $d(\mathfrak{d},\mathfrak{c})\leq 3$. Hence $e(\mathfrak{d})=3$. Now, assume that $n_F(R)=1$.  Then Proposition \ref{distancec} (3) implies that $d(\mathfrak{d},F_1\times I_2\times {\rm Nil}(R_3))=4$ and so we are done.
}
\end{proof}

\begin{lem}\label{diam345}
Let $R$ be an Artinian ring such that $\Gamma_{reg}(R)$ is connected. Then the following statements hold:\\
$\rm(i)$ ${\rm diam}(\Gamma_{reg}(R))=5$ if and only if $n_F(R)=1$.\\
$\rm(ii)$ ${\rm diam}(\Gamma_{reg}(R))=4$ if and only if $n_F(R)=2$.\\
$\rm(iii)$ ${\rm diam}(\Gamma_{reg}(R))=3$ if and only if $n_F(R)\geq 3$.
\end{lem}
\begin{proof}
{Since $\Gamma_{reg}(R)$ is connected, \cite[Theorem 2.1]{actamathhungar} implies that $R\cong F_1\times R_2\times R_3$, where $F_1$ is a field, $R_2$ is an Artinian local ring which is not field and $R_3$ is an Artinian ring. Thus the assertion  follows from \cite[Theorem 2.10]{khashayarmanesh} and this fact that in any Artinian ring $S$, $Z({\rm Nil}(S))=Z(S)$ if and only if $S$ contains no field as its direct summand.}
\end{proof}

From Lemmas \ref{eccentricity3} and \ref{diam345}, we have the following immediate corollary.
\begin{cor}\label{F3ecc3}
Let  $R$ be an Artinian non-reduced ring and $\Gamma_{reg}(R)$ be a connected graph. If $n_F(R)\geq 3$, then for every vertex $I$ of $\Gamma_{reg}(R)$, $e(I)= 3$.
\end{cor}

\begin{lem}\label{reg2}
Let $I$ be an ideal of the Artinian ring $R$. Then $\mathcal{C}^+(I)=\varnothing$ if and only if $I\subseteq {\rm Nil}(R)$.
\end{lem}
\begin{proof}
{First suppose that $I\subseteq {\rm Nil}(R)$. We show that $\mathcal{C}^+(I)=\varnothing$.
Suppose to the contrary $J\in\mathcal{C}^+(I)\neq\varnothing$.
Then $I$ contains a $J$-regular element, say $x$. Choose a non-zero
element $y\in J$. Then there exists a positive integer $n$ such
that $x^ny=0$ and $x^{n-1}y\neq 0$. Since $x^{n-1}y\in J$, we
deduce that $x$ is not $J$-regular, a contradiction. Conversely, suppose that $\mathcal{C}^+(I)=\varnothing$. By \cite[Theorem 8.7]{ati}, $R\cong R_1\times R_2\times\cdots\times R_n$, where $n$ is a positive integer and every $R_i$ is an Artinian local ring. Thus $I=I_1\times I_2\times\cdots\times I_n$, where every $I_i$ is an ideal of $R_i$. Since $\mathcal{C}^+(I)=\varnothing$, we deduce that every $I_i$ is a proper ideal of $R_i$. Hence  $I\subseteq {\rm Nil}(R)$.
}
\end{proof}

According to Corollary \ref{F3ecc3}, we only need to calculate the eccentricity of the vertices of $R$, when $n_F(R)\leq 2$. So from now, we focuss on a ring $R$ which contains at most two fields as its direct summands.
\begin{thm}\label{c-eccentricity}
Let  $R$ be an Artinian non-reduced ring and $\Gamma_{reg}(R)$ be a connected graph. If $n_F(R)\leq 2$ and $\mathfrak{c}=F_1\times I_2\times I_3\in\mathfrak{C}$ is a vertex of $\Gamma_{reg}(R)$, then the following statements hold:\\
$\rm(i)$ If $I_2=R_2$, then
$e(\mathfrak{c})=
  3$.\\
$\rm(ii)$ If $I_2$ is a proper ideal of $R_2$ and $I_3= (0)$, then
$$e(\mathfrak{c})=
  \begin{cases}
   3;       &  n_F(R)=2 \ {\rm and \ either}\ (R_3\ {\rm is\ not \ a \ field)\ or}\ ( R_3\ {\rm is\ a\ field\ and}\ I_2=(0))\\
   4;       & {\rm Otherwise}.
     \end{cases}$$
$\rm(iii)$ If $I_2=(0)$ and $I_3$ is a non-trivial ideal of $R_3$, then
$$e(\mathfrak{c})=
  \begin{cases}
   3;       &  {\rm either}\ n_F(R)=2\ {\rm or}\ I_3\nsubseteq {\rm Nil}(R_3) \\
   4;       & n_F(R)=1\ {\rm and}\ I_3\subseteq {\rm Nil}(R_3).
     \end{cases}$$
$\rm(iv)$ If $I_2$ is a non-trivial ideal of $R_2$ and $I_3=R_3$, then $e(\mathfrak{c})=3$.\\
$\rm(v)$ Let $I_2$ and $I_3$ be non-trivial ideals of $R_2$ and $R_3$, respectively. Then
\begin{itemize}
\item[\rm (a)] If $n_F(R)=1$, then $e(\mathfrak{c})=3$ if and only if $\mathcal{C}^+(I_3)\neq \varnothing$.
\item[\rm (b)] If $n_F(R)=2$, then $R_3\cong T_3\times T_4\times\cdots\times T_n$, where $T_3$ is a field and every $T_i$, $i\neq 4$, is an Artinian local ring which is not field. Moreover, $e(\mathfrak{c})\neq 3$ if and only if
    $I_3=(0)\times Q_4\times\cdots\times Q_n$, where every $Q_i$ is a non-trivial ideal of $T_i$.
\end{itemize}
\end{thm}
\begin{proof}
{(i) Let $\mathfrak{c}=F_1\times R_2\times I_3$, where $I_3$ is a non-trivial ideal of $R_3$. If $x\in \{\mathfrak{a},\mathfrak{b},\mathfrak{d},\mathfrak{u},\mathfrak{v}\}\cup\mathcal{C}$, then Propositions \ref{distancesimple}, \ref{distanceccww} and \ref{distancec} imply that
$d(\mathfrak{c},x)\leq 3$. Also, the existence of the path
$$\mathfrak{c}=F_1\times R_2\times I_3\longrightarrow (0)\times R_2\times (0)\longleftarrow (0)\times R_2\times R_3 \longrightarrow (0)\times K_2\times K_3=\mathfrak{w}$$ shows that $d(\mathfrak{c},\mathfrak{w})\leq 3$, for every vertex $\mathfrak{w}\in\mathfrak{W}$. Thus by Lemma \ref{eccentricity3}, $e(\mathfrak{c})=3$.\\
(ii) Let $\mathfrak{c}=F_1\times I_2\times (0)$, where $I_2$ is a proper ideal of $R_2$. Then by Propositions \ref{distancesimple}, \ref{distanceccww} and \ref{distancec}, we have $d(\mathfrak{c},x)\leq 3$, for every vertex $x\in\{\mathfrak{a},\mathfrak{b},\mathfrak{d},\mathfrak{u},\mathfrak{v}\}\cup\mathfrak{C}$. Therefore, from  Proposition \ref{distancecandw}, we deduce that $3\leq e(\mathfrak{c})\leq 4$. We follow the proof in the following cases:

Case 1. $n_F(R)=1$. In this case, $R\cong F_1\times R_2\times\cdots\times R_n$, where $n\geq 3$ and for every $2\leq i\leq n$, $R_i$ is an Artinian local ring which is not a field. We prove that $d(\mathfrak{c}, {\rm Nil}(R))=4$. Suppose to the contrary, $d(\mathfrak{c},{\rm Nil}(R))\neq 4$. By Proposition \ref{distancecandw}, $d(\mathfrak{c}, {\rm Nil}(R))=3$. Thus Remark \ref{transitivity} implies that there exist two vertices, say $J=J_1\times J_2\times J_3$ and $L=L_1\times L_2\times L_3$, such that one of the following paths exists:
\begin{eqnarray}\label{17}
\mathfrak{c}=F_1\times I_2\times (0)\longleftarrow J_1\times J_2\times J_3 \longrightarrow L_1\times L_2\times L_3 \longleftarrow {\rm Nil}(R)
\end{eqnarray}
\begin{eqnarray}\label{18}
\mathfrak{c}=F_1\times I_2\times (0)\longrightarrow J_1\times J_2\times J_3\longleftarrow L_1\times L_2\times L_3\longrightarrow {\rm Nil}(R)
\end{eqnarray}
By Lemma
\ref{reg2},  Path (\ref{17}) does not exist. So, we can assume that Path (\ref{18}) exists. Thus Remark \ref{productadjacency} implies that $J=F_1\times (0)\times (0)$ and $L=(0)\times R_2\times R_3$ which contradicts the adjacency of $J$ and $L$. Therefore, $e(\mathfrak{c})= 4$.

Case 2. $n_F(R)=2$, $R_3$ is a field and $I_2\neq (0)$. In this case, a similar proof to that of case 1  shows that $d(\mathfrak{c},(0)\times I_2\times R_3)=4$. Thus, in this case, $e(\mathfrak{c})= 4$.

Case 3. $n_F(R)=2$, $R_3$ is a field and $I_2=(0)$. Choose a vertex $\mathfrak{w}\in\mathfrak{W}$. Then there exists a non-trivial ideal $K_2$ of $R_2$ such that either $\mathfrak{w}=(0)\times K_2\times (0)$ or $\mathfrak{w}=(0)\times K_2\times R_3$. If $\mathfrak{w}=(0)\times K_2\times (0)$, then the vertex $F_1\times R_2\times (0)$ is adjacent to both $\mathfrak{c}$ and $\mathfrak{w}$ and so $d(\mathfrak{c},\mathfrak{w})=2$. Also, the existence of the path
\begin{eqnarray*}
\mathfrak{c}=F_1\times (0)\times (0)\longleftarrow F_1\times (0)\times R_3 \longrightarrow (0)\times (0)\times R_3 \longleftarrow (0)\times K_2\times R_3=\mathfrak{w}
\end{eqnarray*}
implies that $d(\mathfrak{c},\mathfrak{w})\leq 3$. Thus, by Lemma \ref{eccentricity3}, $e(\mathfrak{c})= 3$.

Case 4. $n_F(R)=2$ and $R_3$ is not a field. In this case, $R_3\cong T_3\times T_4\times\cdots\times T_n$, where every $T_i$, $i\neq 3$, is an Artinian local ring which is not field and $T_3$ is a field. So, every vertex $\mathfrak{w}\in \mathfrak{W}$ is of the form $(0)\times Q_2\times Q_3\times \cdots\times Q_n$. Now, in the following two subcases, we prove that $d(\mathfrak{c},\mathfrak{w})\leq3$:

Subcase 1. There exists $3\leq j\leq n$ such that $Q_j=T_j$. With no loss of generality, one can assume that $Q_3=T_3$ and so, by Remark \ref{productadjacency}, the path
\begin{eqnarray*}
\mathfrak{c}\longleftarrow F_1\times R_2\times T_3\times (0)\times\cdots\times (0) \longrightarrow (0)\times (0)\times T_3 \times (0)\times\cdots\times (0)\longleftarrow \mathfrak{w}
\end{eqnarray*}
exists and hence $d(\mathfrak{c},\mathfrak{w})\leq3$.

Subcase 2. For every $3\leq j\leq n$, $Q_j\neq T_j$. In this subcase, $Q_3=(0)$ and the existence of the path
\begin{eqnarray*}
\mathfrak{c}\longrightarrow F_1\times (0)\times (0)\times\cdots\times (0) \longleftarrow F_1\times R_2 \times (0)\times T_4\times\cdots\times T_n\longrightarrow \mathfrak{w}
\end{eqnarray*}
shows that $d(\mathfrak{c},\mathfrak{w})\leq3$.

Therefore, in this case, $e(\mathfrak{c})= 3$ and this completes the proof of (ii).\\
(iii) Let $\mathfrak{c}=F_1\times (0)\times I_3$, where $I_3$ is a non-trivial ideal of $R_3$. Then by Propositions \ref{distancesimple},\ref{distanceccww} and \ref{distancec}, we only need to check $d(\mathfrak{c},\mathfrak{w})$, where $\mathfrak{w}\in\mathfrak{W}$. Now, consider the following cases:

Case 1. $n_F(R)=1$ and $I_3\nsubseteq {\rm Nil}(R_3)$. In this case, Lemma \ref{reg2} implies that $\mathcal{C}^+(I_3)\neq \varnothing$. Choose $J_3\in\mathcal{C}^+(I_3)$. Then for every $\mathfrak{w}=(0)\times K_2\times K_3\in\mathfrak{W}$, the path
\begin{eqnarray*}
\mathfrak{c}=F_1\times (0)\times I_3\longrightarrow  (0)\times (0) \times J_3 \longleftarrow (0)\times R_2 \times R_3\longrightarrow \mathfrak{w}
\end{eqnarray*}
exists and so $d(\mathfrak{c},\mathfrak{w})\leq 3$, for every $\mathfrak{w}\in\mathfrak{W}$. Therefore, in this case, $e(\mathfrak{c})=3$

Case 2. $n_F(R)=1$ and $I_3\subseteq {\rm Nil}(R_3)$. Since $I_3\subseteq {\rm Nil}(R_3)$, Lemma \ref{reg2} implies that $\mathcal{C}^+(I_3)=\varnothing$.
By \cite[Theorem 8.7]{actamathhungar}, $R\cong F_1\times R_2\times\cdots\times R_n$, where $n\geq 3$ and for every $2\leq i\leq n$, $R_i$ is an Artinian local ring which is not a field. Thus by \cite[Theorem 2.1]{actamathhungar}, $\mathcal{C}^+({\rm Nil}(R_2\times R_3\times\cdots\times R_n))\cup\mathcal{C}^-({\rm Nil}(R_2\times R_3\times\cdots\times R_n))=\varnothing$. We prove that $d(\mathfrak{c}, {\rm Nil}(R))=4$. Suppose to the contrary, $d(\mathfrak{c},{\rm Nil}(R))\neq 4$. By Proposition \ref{distancecandw}, $d(\mathfrak{c}, {\rm Nil}(R))=3$. Thus Remark \ref{transitivity} implies that there exist two vertices, say $J=J_1\times J_2\times J_3$ and $L=L_1\times L_2\times L_3$, such that one of the following paths exists:
\begin{eqnarray}\label{19}
\mathfrak{c}=F_1\times (0)\times I_3\longleftarrow J_1\times J_2\times J_3 \longrightarrow L_1\times L_2\times L_3 \longleftarrow {\rm Nil}(R)
\end{eqnarray}
\begin{eqnarray}\label{20}
\mathfrak{c}=F_1\times (0)\times I_3\longrightarrow J_1\times J_2\times J_3\longleftarrow L_1\times L_2\times L_3\longrightarrow {\rm Nil}(R)
\end{eqnarray}
By Lemma
\ref{reg2},  Path (\ref{19}) does not exist. So, we can assume that Path (\ref{20}) exists. Thus Remark \ref{productadjacency} implies that $J=F_1\times (0)\times (0)$. So,  $L=F_1\times R_2\times R_3$, a contradiction. Thus $d(\mathfrak{c}, {\rm Nil}(R))=4$ and hence $e(\mathfrak{c})= 4$.

Case 3. $n_F(R)=2$ and $R_3$ is not a field. In this case, $R_3\cong T_3\times T_4\times\cdots\times T_n$, where every $T_i$, $i\neq 3$, is an Artinian local ring which is not field and $T_3$ is a field. So, every vertex $\mathfrak{w}\in \mathfrak{W}$ is of the form $(0)\times Q_2\times Q_3\times \cdots\times Q_n$. Now, in the following two subcases, we prove that $d(\mathfrak{c},\mathfrak{w})\leq3$:

Subcase 1. There exists $3\leq j\leq n$ such that $Q_j=T_j$. With no loss of generality, one can assume that $Q_3=T_3$ and so, by Remark \ref{productadjacency}, the path
\begin{eqnarray*}
\mathfrak{c}\longleftarrow F_1\times (0)\times T_3\times T_4\times\cdots\times T_n \longrightarrow (0)\times (0)\times T_3 \times (0)\times\cdots\times (0)\longleftarrow \mathfrak{w}
\end{eqnarray*}
exists and hence $d(\mathfrak{c},\mathfrak{w})\leq3$.

Subcase 2. For every $3\leq j\leq n$, $Q_j\neq T_j$. In this subcase, the existence of the path
\begin{eqnarray*}
\mathfrak{c}\longrightarrow F_1\times (0)\times (0)\times\cdots\times (0) \longleftarrow F_1\times  (0) \times T_3\times\cdots\times T_n\longrightarrow \mathfrak{w}
\end{eqnarray*}
shows that $d(\mathfrak{c},\mathfrak{w})\leq3$.

Therefore, in this case, Lemma \ref{eccentricity3} implies that $e(\mathfrak{c})= 3$.\\
(iv) Assume that $\mathfrak{c}= F_1\times I_2\times R_3$, where $I_2$ is a non-trivial ideal of $R_2$.  By Propositions \ref{distancesimple},\ref{distanceccww} and \ref{distancec}, we have $d(\mathfrak{c},x)\leq 3$, for every vertex $x\in\{\mathfrak{a},\mathfrak{b},\mathfrak{d},\mathfrak{u},\mathfrak{v}\}\cup\mathfrak{C}$. Also, for every vertex $\mathfrak{w}=(0)\times K_2\times K_3\in\mathfrak{W}$, the path
\begin{eqnarray*}
\mathfrak{c} \longrightarrow (0)\times (0)\times R_3 \longleftarrow (0)\times R_2 \times R_3 \longrightarrow \mathfrak{w}
\end{eqnarray*}
exists and hence by Lemma \ref{eccentricity3}, $e(\mathfrak{c})=3$.\\
(v) Let $\mathfrak{c}=F_1\times I_2\times I_3$, where $I_2$ and $I_3$ are non-trivial ideals of $R_2$ and $R_3$, respectively. Then by Propositions \ref{distancesimple},\ref{distanceccww} and \ref{distancec}, we have $d(\mathfrak{c},x)\leq 3$, for every vertex $x\in\{\mathfrak{a},\mathfrak{b},\mathfrak{u},\mathfrak{v}\}\cup\mathfrak{C}$. Hence by Lemma \ref{eccentricity3}, $e(\mathfrak{c})$ depends only on $d(\mathfrak{c},\mathfrak{d})$ and $d(\mathfrak{c},\mathfrak{w})$, where $\mathfrak{w}\in \mathfrak{W}$.  On the other hand,  Proposition \ref{distancec}(2) implies that $d(\mathfrak{c},\mathfrak{d})\leq 3$ if and only if $\mathcal{C}^+(I_3)\cup\mathcal{C}^-(I_3)\neq \varnothing$. Therefore, by Lemma \ref{eccentricity3}, $e(\mathfrak{c})=3$ if and only if $\mathcal{C}^+(I_3)\cup\mathcal{C}^-(I_3)\neq \varnothing$ and $d(\mathfrak{c},\mathfrak{w})\leq 3$, for every $\mathfrak{w}\in\mathfrak{W}$. To complete the proof, we consider the following three cases:

Case 1. $\mathcal{C}^+(I_3)\neq\varnothing$. In this case, choose $J_3\in\mathcal{C}^+(I_3)$. Then for every $\mathfrak{w}=(0)\times K_2\times K_3\in\mathfrak{W}$, the path
\begin{eqnarray*}
\mathfrak{c}=F_1\times I_2\times I_3\longrightarrow (0)\times (0)\times J_3\longleftarrow (0)\times R_2\times R_3\longrightarrow (0)\times K_2\times K_3=\mathfrak{w}
\end{eqnarray*}
exists and so $e(\mathfrak{c})=3$.

Case 2. $\mathcal{C}^+(I_3)=\varnothing$ and $n_F(R)=1$. In this case, we prove that $d(\mathfrak{c},(0)\times I_2\times {\rm Nil}(R_3))\geq 4$. Suppose to the contrary, $d(\mathfrak{c},\mathfrak{w})\leq 3$. Then Remark \ref{transitivity} implies that there are two vertices, say $J=J_1\times J_2\times J_3$ and $L=L_1\times L_2\times L_3$, such that one of the following paths exists:
\begin{eqnarray}\label{21}
\mathfrak{c}=F_1\times I_2\times I_3\longleftarrow J_1\times J_2\times J_3 \longrightarrow L_1\times L_2\times L_3 \longleftarrow (0)\times I_2\times {\rm Nil}(R_3)
\end{eqnarray}
\begin{eqnarray}\label{22}
\mathfrak{c}=F_1\times I_2\times I_3\longrightarrow J_1\times J_2\times J_3\longleftarrow L_1\times L_2\times L_3\longrightarrow (0)\times I_2\times {\rm Nil}(R_3)
\end{eqnarray}
By Lemma
\ref{reg2},  Path (\ref{21}) does not exist. So, we can assume that Path (\ref{22}) exists. Thus Remark \ref{productadjacency} implies that $J=F_1\times(0)\times(0)$ and $L=(0)\times R_2\times R_3$ which contradicts the adjacency of $J$ and $L$. Therefore, $e(\mathfrak{c})= 4$.

Case 3. $\mathcal{C}^+(I_3)=\varnothing$ and $n_F(R)=2$. In this case, Lemma \ref{k1C+-} implies that $\mathcal{C}^-(I_3)\neq\varnothing$. Since $I_3$ is a non-trivial ideal of $R_3$ and $n_F(R)=2$, we deduce that $R_3\cong T_3\times T_4\times\cdots\times T_n$, where $T_3$ is a field and for every $4\leq i\leq n$, $T_i$ is an Artinian local ring which is not field. So, $I_3=(0)\times Q_4\times\cdots\times Q_n$, where every $Q_i$ is a proper ideal of $T_i$. To complete the proof, we consider the following two subcases:

Subcase 1. Every $Q_j$, $4\leq j\leq n$, is a non-trivial ideal of $T_j$. In this subcase, we prove that $d(\mathfrak{c},(0)\times I_2\times I_3^c)\geq 4$ (Note that $I_3^c=T_3\times Q_4\times\cdots\times Q_n$). Suppose to the contrary, $d(\mathfrak{c},(0)\times I_2\times I_3^c)\leq 3$. Then by Remark \ref{transitivity}, there are two vertices, say $J=J_1\times J_2\times J_3$ and $L=L_1\times L_2\times L_3$, such that one of the following paths exists:
\begin{eqnarray}\label{23}
\mathfrak{c}=F_1\times I_2\times I_3\longleftarrow J_1\times J_2\times J_3 \longrightarrow L_1\times L_2\times L_3 \longleftarrow (0)\times I_2\times I_3^c
\end{eqnarray}
\begin{eqnarray}\label{24}
\mathfrak{c}=F_1\times I_2\times I_3\longrightarrow J_1\times J_2\times J_3\longleftarrow L_1\times L_2\times L_3\longrightarrow (0)\times I_2\times I_3^c
\end{eqnarray}
If Path (\ref{23}) exists, then by Remark \ref{productadjacency},  $J=F_1\times R_2\times(0)\times T_4\times\cdots\times T_n$ and $L=(0)\times(0)\times T_3\times (0)\times\cdots\times (0)$ and this contradicts the adjacency of $J$ and $L$. Also, since $\mathcal{C}^+(I_3)=\varnothing$, by Remark \ref{productadjacency}, the existence of the Path (\ref{24}) implies that   $J=F_1\times(0)\times (0)$ and $L=(0)\times R_2\times R_3$, a contradiction. Hence $e(\mathfrak{c})\geq 4$

Subcase 2. There exists $4\leq j\leq n$ such that $Q_j=(0)$. With no loss of generality, one can assume that $Q_4=(0)$. We show that for every $\mathfrak{w}=(0)\times K_2\times K_3\in\mathfrak{W}$, $d(\mathfrak{c},\mathfrak{w})\leq 3$. From Lemma \ref{k1C+-}, we deduce that $\mathcal{C}^+(K_3)\cup\mathcal{C}^-(K_3)\neq\varnothing$. If $\mathcal{C}^-(K_3)\neq\varnothing$, then there exists a non-trivial ideal $L_3\in\mathcal{C}^-(K_3)$. Thus the path
\begin{eqnarray*}
\mathfrak{c}=F_1\times I_2\times I_3\longrightarrow F_1\times (0)\times (0)\longleftarrow F_1\times R_2\times L_3\longrightarrow (0)\times K_2\times K_3=\mathfrak{w}
\end{eqnarray*}
exists and so there is no thing to prove. Thus we can suppose that $\mathcal{C}^-(K_3)=\varnothing$ and so $K_3=T_3\times Q_4\times\cdots\times Q_n$, where $Q_i\neq (0)$, for every $4\leq i\leq n$. Setting $J_3=T_3\times (0)\times T_5\times \cdots\times T_n$ and $L_3=T_3\times (0)\times \cdots\times (0)$, Remark \ref{productadjacency} implies that the path
\begin{eqnarray*}
\mathfrak{c}=F_1\times I_2\times I_3\longleftarrow F_1\times R_2\times J_3\longrightarrow (0)\times (0)\times L_3\longleftarrow (0)\times K_2\times K_3=\mathfrak{w}
\end{eqnarray*}
exists and so $d(\mathfrak{c},\mathfrak{w})\leq 3$. Therefore, in this subcase, $e(\mathfrak{c})=3$. So, we are done.
}
\end{proof}

Let $R$ be an Artinian ring and $I$ and $J$ be two non-trivial ideals of $R$. Then it is clear that
 $I\in\mathcal{C}^+(J)$ if and only if $I^c\in\mathcal{C}^-(J^c)$. Moreover, $I\longrightarrow J$ is an arc of $\overrightarrow{\Gamma_{reg}}(R)$ if and only if $J^c\longrightarrow I^c$ of $\overrightarrow{\Gamma_{reg}}(R)$.
Thus $d(I,J)=d(I^c,J^c)$ and hence $e(I)=e(I^c)$. Moreover, we have $I\in\mathfrak{C}$ if and only if $I^c\in\mathfrak{W}$. Thus using this facts and applying the similar proof to that of Theorem \ref{c-eccentricity}, one
can prove the following theorem.

\begin{thm}\label{w-eccentricity}
Let  $R$ be an Artinian non-reduced ring such that $\Gamma_{reg}(R)$ be a connected graph. If $n_F(R)\leq 2$ and $\mathfrak{w}=(0)\times K_2\times K_3\in\mathfrak{W}$ is a vertex of $\Gamma_{reg}(R)$, then the following statements hold:\\
$\rm(i)$ If $K_2=(0)$, then
$e(\mathfrak{w})=
  3$.\\
$\rm(ii)$ If $K_2$ is a non-trivial ideal of  $R_3$ and $K_3= R_3$, then
$$e(\mathfrak{w})=
  \begin{cases}
   3;       &  n_F(R)=2\ {\rm and}\ R_3\ {\rm is\ not \ a\ field} \\
   4;       & {\rm Otherwise}.
     \end{cases}$$
$\rm(iii)$ If $K_2=R_2$, then
$$e(\mathfrak{w})=
  \begin{cases}
   3;       &  {\rm Either}\ n_F(R)=2\ {\rm or}\ \mathcal{C}^-(K_3)\neq\varnothing\\
   4;       & {\rm Otherwise}.
     \end{cases}$$
$\rm(iv)$ If $K_2$ is a non-trivial ideal of $R_2$ and $K_3=(0)$, then $e(\mathfrak{w})=3$.\\
$\rm(v)$ Let $K_2$ and $K_3$ be non-trivial ideals of $R_2$ and $R_3$, respectively. Then
\begin{itemize}
\item[\rm (a)] If $n_F(R)=1$, then $e(\mathfrak{w})=3$ if and only if $\mathcal{C}^-(K_3)\neq\varnothing$.
\item[\rm (b)] If $n_F(R)=2$, then $R_3\cong T_3\times T_4\times\cdots\times T_n$, where $T_3$ is a field and every $T_i$, $i\neq 4$, is an Artinian local ring which is not field. Moreover, $e(\mathfrak{w})\neq 3$ if and only if
    $K_3=T_3\times Q_4\times\cdots\times Q_n$, where every $Q_i$ is a non-trivial ideal of $T_i$.
\end{itemize}
\end{thm}



Finally, we prove the main theorem of this paper.
\begin{thm}\label{radius}
If $R$ is an Artinian ring and $\Gamma_{reg}(R)$ is connected, then $r(\Gamma_{reg}(R))=3$.
\end{thm}
\begin{proof}
{Let $\Gamma_{reg}(R)$ be a connected  graph. Then \cite[Theorem 2.3]{actamathhungar} implies that $|{\rm Max}(R)|\geq 3$ and $R$ contains a field as its
direct summand. Since $R$ has at least three maximal ideals, \cite[Theorem 8.7]{ati} implies that $R\cong F_1\times R_2\times R_3$, where $F_1$ is a field and $R_2$ and $R_3$ are Artinian rings. First suppose that $R$ is reduced. Then $R\cong F_1\times \cdots\times F_n$, where $n\geq 3$ and every $F_i$ is a field. For every ideal $I=I_1\times \cdots\times I_n$ of $R$, define $$\Delta_I=\{k\,|\ 1\leq k\leq n\,\,
{\rm and} \,\,I_k=F_k\}$$ and $$\Omega=\{\Delta_I\,|\ I\ {\rm is \ a\ non-trivial\ ideal\ of} \ R\}.$$
Clearly, $\Delta_I=\Delta_J$ if and only if $I=J$. Thus there is a one to one correspondence between $\Omega$ and the set of proper and
non-empty subsets of $\{1,\ldots,n\}$. We claim that $d(I,J)\leq 3$, for every two distinct vertices $I$ and $J$ of $\Gamma_{reg}(R)$. To see this, we consider the following two cases:

Case 1. $\Delta_I\cap\Delta_J=\varnothing$. Since $n\geq 3$, by pigeon-hole principle and with no loss of generality, we can assume that $|\Delta_I|\geq |\Delta_J|$, $|\Delta_I|\geq 2$ and $1\in \Delta_I$. Let $I_1$ and $I_2$ be two vertices of $\overrightarrow{\Gamma_{reg}}(R)$ such that $\Delta_{I_1}=\{1\}$ and $\Delta_{I_2}=A_J\cup\{1\}$. Then there is the path
$I\longrightarrow I_1\longleftarrow I_2\longrightarrow J$ in $\overrightarrow{\Gamma_{reg}}(R)$ and hence $d(I,J)\leq 3$.

Case 2. $\Delta_I\cap\Delta_J\neq\varnothing$.
If either $\Delta_I\subset\Delta_J$ or $\Delta_J\subset\Delta_I$, then $I$ and $J$ are adjacent. So, we can assume that neither $\Delta_I\nsubseteq\Delta_J$ nor $\Delta_J\nsubseteq\Delta_I$. Choose $i\in \Delta_I\cap\Delta_J$. Then it is clear that $I_i=(0)\times\cdots\times F_i\times\cdots\times (0)$ is adjacent to both $I$ and $J$ and so $d(I,J)=2$.

So, the claim is proved. Thus from Lemma \ref{eccentricity3}, we deduce that $e(I)=3$. Therefore, in any case, $r(\Gamma_{reg}(R))=3$.
Now, suppose that $R$ is non-reduced. Then the assertion follows from Theorems \ref{abduv-eccentricity}, \ref{c-eccentricity}, \ref{w-eccentricity} and Corollary \ref{F3ecc3}.
}
\end{proof}

Using Theorems \ref{abduv-eccentricity}, \ref{c-eccentricity} and \ref{w-eccentricity}, we have the following immediate corollary in which  the center of $\Gamma_{reg}(R)$ is determined.

\begin{cor}\label{center}
Let  $R$ be an Artinian non-reduced ring and $\Gamma_{reg}(R)$ be a connected graph. Then the following statements hold:
\begin{enumerate}
\item[\rm(i)] If $n_F(R)=1$ and
$X=\{F_1\times I_2\times I_3|\ {\rm either}\ I_2=R_2\ {\rm or}\ I_3\nsubseteq {\rm Nil}(R_3)\}$, then the center of $\Gamma_{reg}(R)$ equals $X\cup X^c$.
\item[\rm(ii)] If $n_F(R)=2$ and $|{\rm Max}(R)|=3$, then every vertex of $\Gamma_{reg}(R)$ is centeral.
\item[\rm(iii)] If $n_F(R)=2$ and $|{\rm Max}(R)|\geq 4$, then $R\cong F_1\times T_2\times F_3\times T_4\times \cdots\times T_n$, where $F_1$ and $F_3$ are fields and every $T_i$ is an Artinian local ring which is not field. Moreover, the vertex $I$ is central if and only if neither
$I=F_1\times I_2\times (0) \times I_4\times\cdots\times I_n$ nor $I=(0)\times I_2\times F_3 \times I_4\times\cdots\times I_n$, for every non-trivial ideals $I_i$ of $T_i$.
\item[\rm(iv)] If $n_F(R)\geq 3$, then every vertex of $\Gamma_{reg}(R)$ is central.
\end{enumerate}
\end{cor}


{}

\end{document}